\documentclass[12pt]{amsart}
\usepackage{amsmath,latexsym,amsfonts,amssymb,amsthm}
\usepackage{geometry}
\usepackage{hyperref}
\usepackage{mathrsfs}
\usepackage{graphicx,color}
\usepackage{tikz-cd}
\usepackage{tikz}
\usetikzlibrary{positioning}
\usepackage{mathtools}
\usepackage[all,cmtip]{xy}

\newtheorem{prop}{Proposition}[section]
\newtheorem{thm}[prop]{Theorem}
\newtheorem{cor}[prop]{Corollary}

\newtheorem{lem}[prop]{Lemma}

\theoremstyle{definition}
\newtheorem{que}[prop]{Question}
\newtheorem{defn}[prop]{Definition}

\newtheorem*{claim*}{Claim}

\newcommand{\bP}{\mathbb{P}}
\newcommand{\bC}{\mathbb{C}}

\newcommand{\bQ}{\mathbb{Q}}
\newcommand{\bZ}{\mathbb{Z}}

\newcommand{\cO}{\mathcal{O}}

\newcommand{\cM}{\mathcal{M}}

\newcommand{\cJ}{\mathcal{J}}

\newcommand{\mult}{\mathrm{mult}}
\newcommand{\lct}{\mathrm{lct}}

\newcommand{\vol}{\mathrm{vol}}
\newcommand{\ord}{\mathrm{ord}}
\newcommand{\codim}{\mathrm{codim}}

\begin{document}

\title{K-stability of birationally superrigid Fano varieties}

\author{Charlie Stibitz}
\address{Department of Mathematics, Princeton University, Princeton, NJ, 08544-1000.}
\email{cstibitz@math.princeton.edu}

\author{Ziquan Zhuang}
\address{Department of Mathematics, Princeton University, Princeton, NJ, 08544-1000.}
\email{zzhuang@math.princeton.edu}

\date{}

\maketitle

\begin{abstract}
    We prove that every birationally superrigid Fano variety whose alpha invariant is greater than (resp. no smaller than) $\frac{1}{2}$ is K-stable (resp. K-semistable). We also prove that the alpha invariant of a birationally superrigid Fano variety of dimension $n$ is at least $\frac{1}{n+1}$ (under mild assumptions) and that the moduli space (if exists) of birationally superrigid Fano varieties is separated.
\end{abstract}

\section{Introduction}

The notion of birational superrigidity was introduced as a generalization of Iskovskikh and Manin's work \cite{IM-quartic} on the non-rationality of quartic threefolds; on the other hand, the concept of K-stability emerges in the study of K\"ahler-Einstein metrics on Fano manifolds. While the two notions have different nature of origin, they seem to resemble each other in the following sense: it is well known that a Fano variety $X$ of Picard number $1$ is birationally superrigid if and only if $(X,M)$ has canonical singularities for every movable boundary  $M\sim_\bQ -K_X$; on the other hand, by the recent work of \cite{Fujita-delta,BJ-delta}, the K-(semi)stability of $X$ is (roughly speaking) characterized by the log canonicity of basis type divisors, which is the average of a basis of some pluri-anticanonical system. In other words, both notions are tied to the singularities of certain anticanonical $\bQ$-divisors and so it is very natural to expect some relation between them. Indeed, the slope stability (a weaker notion of K-stability) of birationally superrigid Fano manifolds has been established by \cite{oo-slope-stability} under some mild assumptions, and it is conjectured \cite{oo-slope-stability,alpha-of-rigid-3fold} that birationally rigid Fano varieties are always K-stable.

In this note, we give a partial solution to this conjecture. Here is our main result:

\begin{defn}
The alpha invariant $\alpha(X)$ of a $\bQ$-Fano variety $X$ (i.e. $X$ has klt singularities and $-K_X$ is ample) is defined as the supremum of all $t>0$ such that $(X,tD)$ is log canonical for every effective $\bQ$-divisor $D\sim_\bQ -K_X$.
\end{defn}

\begin{thm} \label{thm:main}
Let $X$ be a $\bQ$-Fano variety of Picard number $1$. If $X$ is birationally superrigid $($or more generally, $(X,M)$ is log canonical for every movable boundary $M\sim_\bQ -K_X)$ and $\alpha(X)\ge \frac{1}{2}$ $($resp. $>\frac{1}{2})$, then $X$ is K-semistable $($resp. K-stable$)$.
\end{thm}

It is well known that smooth Fano hypersurfaces of index $1$ and dimension $n\ge3$ are birationally superrigid \cite{IM-quartic,dFEM-bounds-on-lct,dF-hypersurface} and their alpha invariants are at least $\frac{n}{n+1}$ \cite{C-hypersurface-lct}, hence we have the following immediate corollary, reproving the K-stability of Fano hypersurfaces of index one:

\begin{cor} \cite{Fujita-alpha}
Let $X\subseteq\bP^n$ be a smooth hypersurface of degree $n\ge 4$, then $X$ is K-stable.
\end{cor}

Another application is to the K-stability of general index two hypersurfaces. By \cite{C-hypersurface-lct} and \cite{Puk-index-two}, such hypersurfaces have alpha invariant $\frac{1}{2}$ and $(X,M)$ is lc for every movable boundary $M\sim_\bQ -K_X$. Hence by analyzing the equality case in Theorem \ref{thm:main}, we prove:

\begin{cor} \label{cor:index two}
Let $n\ge 16$ and let $U\subseteq \bP H^0(\bP^{n+1},\cO_{\bP^{n+1}}(n))$ be the parameter space of smooth index two hypersurfaces. Let $T\subseteq U$ be the set of hypersurfaces that are not K-stable. Then $\codim_U T\ge \frac{1}{2}(n-11)(n-10)-10$.
\end{cor}

Although alpha invariant is in general hard to estimate, the known examples seem to suggest that birationally (super)rigid varieties have large alpha invariants. In view of Theorem \ref{thm:main}, it is therefore natural to ask the following question:

\begin{que} \label{que:alpha>1/2}
Let $X$ be a birationally superrigid Fano variety. Is it true that $\alpha(X) > \frac{1}{2}$?
\end{que}

Obviously, a positive answer to this question will confirm the K-stability of all birationally superrigid Fano varieties. At this point, we only have a weaker estimate:

\begin{thm} \label{thm:alpha>=1/(n+1)}
Let $X$ be a $\bQ$-Fano variety of Picard number $1$ and dimension $n\ge 3$. Assume that $X$ is birationally superrigid $($or more generally, $(X,M)$ is log canonical for every movable boundary $M\sim_\bQ -K_X)$, $-K_X$ generates the class group $\mathrm{Cl}(X)$ of $X$ and $|-K_X|$ is base point free, then $\alpha(X)\ge\frac{1}{n+1}$.
\end{thm}

Note that this is in line with the conjectural K-stability of such varieties, since by \cite[Theorem 3.5]{Fujita-delta}, the alpha invariant of a K-semistable Fano variety is always $\ge \frac{1}{n+1}$. We also remark that the assumptions about the index and base point freeness in the above theorem seem to be mild and they are satisfied by most known example of birationally superrigid varieties.

As another evidence towards a positive answer of Question \ref{que:alpha>1/2} and K-stability of birationally superrigid Fano varieties, we prove:

\begin{thm} \label{thm:separatedness}
Let $f:X\rightarrow C$, $g:Y\rightarrow C$ be two flat families of Fano varieties $($i.e. all geometric fibers are integral, normal and Fano$)$ over a smooth pointed curve $0\in C$. Assume that the central fibers $X_0=f^{-1}(0)$ and $Y_0=g^{-1}(0)$ are birationally superrigid and there exists an isomorphism $\rho:X\backslash X_0 \cong Y\backslash Y_0$ over the punctured curve $C\backslash 0$. Then $\rho$ induces an isomorphism $X\cong Y$ over $C$.
\end{thm}

In other words, the moduli space (if exist) of birationally superrigid Fano varieties is separated. Similar statement is also conjectured for families of K-polystable Fano varieties and our proof of Theorem \ref{thm:separatedness} is indeed inspired by the recent work \cite{BX} in the uniformly K-stable case. One should also note that if the answer to Question \ref{que:alpha>1/2} is positive, then Theorem \ref{thm:separatedness} follows immediately from \cite[Theorem 1.5]{C-alpha>1}.

\subsection*{Acknowledgement}

The authors would like to thank their advisor J\'anos Koll\'ar for constant support, encouragement and numerous inspiring conversations. The first author wishes to thank Harold Blum for useful discussion as well. Both authors also wish to thank Yuchen Liu for helpful discussions.

\section{Preliminary}

\subsection{Notation and conventions}

We work over the field $\bC$ of complex numbers. Unless otherwise specified, all varieties are assumed to be projective and normal and divisors are understood as $\bQ$-divisor. The notions of canonical, klt and log canonical (lc) singularities are defined in the sense of \cite[Definition 3.5]{Kol-sing-of-pairs}. A movable boundary is defined as an expression of the form $a\cM$ where $a\in\bQ$ and $\cM$ is a movable linear system. Its $\bQ$-linear equivalence class is defined in an evident way. If $M=a\cM$ is a movable boundary on $X$, we say that $(X,M)$ is klt (resp. canonical, lc) if for $k\gg 0$ and for general members $D_1,\cdots,D_k$ of the linear system $\cM$, the pair $(X,M_k)$ (where $M_k=\frac{a}{k}\sum_{i=1}^k D_i$) is klt (resp. canonical, lc) in the usual sense. The log canonical threshold \cite[Definition 8.1]{Kol-sing-of-pairs} of a divisor $D$ on $X$ is denoted by $\lct(X;D)$.

\subsection{K-stability}

We refer to \cite{Tian-K-stability-defn, Don-K-stability-defn} for the original definition of K-stability using test configurations. In this paper we use the following equivalent valuative criterion.

\begin{defn}[{\cite[Definition 1.1]{Fujita-valuative-criterion}}] \label{defn:threshold and beta}
Let $X$ be a $\bQ$-Fano variety of dimension $n$. Let $F$ be a prime divisor over $X$, i.e., there exists a projective birational morphism $\pi: Y\to X$ with $Y$ normal such that $F$ is a prime divisor on $Y$.
    \begin{enumerate}
        \item For any $x\ge 0$, we define $\vol_X(-K_X-xF):=\vol_Y(-\pi^*K_X-xF)$.
        \item The \emph{pseudo-effective threshold} $\tau(F)$ of $F$ with respects to $-K_X$ is defined as
        \[\tau(F):=\sup\{\tau>0\,|\, \vol_X(-K_X-\tau F)>0\}.\]
        \item Let $A_X(F)$ be the log discrepancy of $F$ with respect to $X$. We set
        \[\beta(F):=A_X(F)\cdot((-K_X)^n)-\int_0^{\infty}\vol_X(-K_X-xF)\mathrm{d}x.\]
        \item $F$ is said to be \emph{dreamy} if the graded algebra
        \[\bigoplus_{k,j\in\bZ_{\geq 0}}H^0(Y, -kr\pi^*K_X-jF)\]
        is finitely generated for some (hence, for any) $r\in\bZ{>0}$ with $rK_X$ Cartier. 
    \end{enumerate}
\end{defn}

\begin{thm}[{\cite[Theorems 1.3 and 1.4]{Fujita-valuative-criterion} and \cite[Theorem 3.7]{Li-equivariant-minimize}}] \label{thm:beta-criterion}
Let $X$ be a $\bQ$-Fano variety. Then $X$ is K-stable $($resp.\ K-semistable$)$ if and only if $\beta(F)>0$ $($resp.\ $\beta(F)\geq 0)$ holds for any dreamy prime divisor $F$ over $X$. 
\end{thm}

\subsection{Birational superrigidity}

A Fano variety $X$ is said to be birationally superrigid if it has terminal singularities, is $\bQ$-factorial of Picard number one and every birational map $f:X\dashrightarrow Y$ from $X$ to a Mori fiber space is an isomorphism (see e.g. \cite[Definition 1.25]{Noether-Fano}). We have the following equivalent characterization of birational superrigidity.

\begin{thm}[{\cite[Theorem 1.26]{Noether-Fano}}] \label{thm:noether-fano}
Let $X$ be a Fano variety. Then it is birationally superrigid if and only if it has $\bQ$-factorial terminal singularities, Picard number one, and for every movable boundary $M\sim_\bQ -K_X$ on $X$, the pair $(X,M)$ has canonical singularities.
\end{thm}

\section{Proofs}

In this section we prove the results stated in the introduction.

\begin{proof}[Proof of Theorem \ref{thm:main}]
The proof strategy is similar to those of \cite[Proposition 2.1]{Fujita-plt-blowup}. By assumption we have $\lct(X;D)\ge \frac{1}{2}$ (resp. $>\frac{1}{2}$) for every effective divisor $D\sim_\bQ -K_X$. For simplicity we only prove the K-semistability, since the K-stability part is almost identical. As such, we assume $\lct(X;D)\ge \frac{1}{2}$ in the rest of the proof. Let $F$ be a dreamy divisor over $X$ and let $\tau=\tau(F)$. Let $\pi:Y\to X$ be a projective birational morphism such that $F$ is a prime divisor on $Y$ and let 
\[b=\frac{1}{((-K_X)^n)}\int_0^\tau \vol_X(-K_X-xF) \mathrm{d}x.
\]
By Theorem \ref{thm:beta-criterion}, it suffices to show that $b\le A=A_X(F)$. 

Suppose that this is not the case, i.e. $b>A$. As in the proof of \cite[Proposition 2.1]{Fujita-plt-blowup}, we have
\begin{equation} \label{eq:b}
    \int_0^{\tau} (x-b)\cdot \vol_{Y|F}(-\pi^*K_X-xF) \mathrm{d}x = 0
\end{equation}
where $\vol_{Y|F}$ denotes the restricted volume of a divisor to $F$ (see \cite{ELMNP-restricted-volume}). Since $X$ has Picard number one, every divisor on $X$ is linearly equivalent to a multiple of $-K_X$, hence by our assumption that $(X,M)$ is log canonical for every movable boundary $M$, we see that there exists at most one irreducible divisor $D\sim_\bQ -K_X$ such that $\ord_F(D)>A$. It then follows that $\ord_F(D)=\tau$ by the definition of $\tau(F)$ and moreover, as \[-\pi^*K_X-xF=\frac{\tau-x}{\tau-A}(-\pi^*K_X-AF)+\frac{x-A}{\tau-A}(-\pi^*K_X-\tau F),
\]
$D$ appears with multiplicity at least $\frac{x-A}{\tau-A}$ in the stable base locus of $-\pi^*K_X-xF$ for every $x>A$ and we obtain
\begin{equation} \label{eq:x>A}
    \vol_{Y|F}(-\pi^*K_X-xF)=\left(\frac{\tau-x}{\tau-A}\right)^{n-1}\vol_{Y|F}(-\pi^*K_X-AF)
\end{equation}
when $x>A$. On the other hand, by the log-concavity of restricted volume \cite[Theorem A]{ELMNP-restricted-volume}, we have
\begin{equation} \label{eq:x<=A}
    \vol_{Y|F}(-\pi^*K_X-xF)\ge \left(\frac{x}{A}\right)^{n-1}\vol_{Y|F}(-\pi^*K_X-AF)
\end{equation}
whenever $x\in[0,A]$. Since $b>A$, combining \eqref{eq:b}, \eqref{eq:x>A} and \eqref{eq:x<=A} we get the inequality
\begin{equation} \label{eq:integral}
    0\le \int_0^A (x-b)\left(\frac{x}{A}\right)^{n-1} \mathrm{d}x + \int_A^{\tau} (x-b) \left(\frac{\tau-x}{\tau-A}\right)^{n-1} \mathrm{d}x, 
\end{equation}
which is equivalent to
\begin{equation} \label{eq:tau,A,b}
    \frac{\tau-2A}{n(n+1)}+\frac{A-b}{n}\ge 0.
\end{equation}
As $b>A$, we have $\ord_F(D)=\tau>2A$, which implies $\lct(X;D)< \frac{1}{2}$, contradicting our assumption.
\end{proof}

It is not hard to characterize the equality case from the above proof.

\begin{cor} \label{cor:equality}
Let $X$ be as in Theorem \ref{thm:main}. Assume that $X$ is strictly K-semistable, then there exists a dreamy prime divisor $F$ over $X$, a movable boundary $M\sim_\bQ -K_X$ and an effective divisor $D\sim_\bQ -K_X$ such that $F$ is a log canonical place of $(X,M)$ and $(X,\frac{1}{2}D)$. In particular, if $X$ is birationally superrigid and $\alpha(X)\ge \frac{1}{2}$, then $X$ is K-stable.
\end{cor}

\begin{proof}
We keep the notation from the above proof. Let $\eta=\eta(F)$ be movable threshold of $-K_X$ with respect to $F$, i.e. the supremum of $c\ge 0$ such that there exists an effective divisor $D_0\sim_\bQ -K_X$ with $\ord_F(D_0)=c$ whose support does not contain $D$ (see e.g.  \cite[Definition 4.1]{Z-cpi}). Since $F$ is dreamy, this is indeed a maximum and there exists $D_1\sim_\bQ -K_X$ such that $\ord_F(D_1)=\eta$ (one can simply take $D_1$ to be the divisor corresponding to a generator of $\bigoplus_{k,j\in\bZ_{\geq 0}}H^0(Y, -kr\pi^*K_X-jF)$ with largest slope $\frac{j}{kr}$ among those that does not vanish on $D$). 

Suppose that $X$ is strictly K-semistable and choose $F$ to be a dreamy divisor over $X$ such that $\beta(F)=0$, then we have $b=A\ge \eta$. In this case by the same proof as above the (in)equalities \eqref{eq:x>A}, \eqref{eq:x<=A}, \eqref{eq:integral} and \eqref{eq:tau,A,b} holds true with $\eta$ in place of $A$ and we have
\[\frac{\tau-2\eta}{n(n+1)}+\frac{\eta - b}{n}\ge 0,\]
or $(\tau-2A)+(n-1)(\eta - A)\ge 0$ (note that $b=A$). But by assumption we have $\tau\le 2A$ and $\eta\le A$, hence this is only possible when $\eta=A$ and $\tau=2A$. Taking $M$ to be the linear system generated by a sufficiently divisible multiple of $D$ and $D_1$ (and then rescale so that $M\sim_\bQ -K_X$) finishes the proof.
\end{proof}

\begin{proof}[Proof of Corollary \ref{cor:index two}]
Let $S\subseteq U$ be the set of \emph{regular} hypersurfaces as defined in \cite[\S 0.2]{Puk-sing-index-two}. By \cite[Theorem 2]{Puk-sing-index-two}, $S$ is non-empty and the complement of $S$ has codimension at least $\frac{1}{2}(n-11)(n-10)-10$. Therefore, it suffices to show that every hypersurface in the set $S$ is K-stable. Let $X$ be such a hypersurface.

Let $H$ be the hyperplane class and let $D\sim_\bQ H\sim_\bQ -\frac{1}{2}K_X$ be an effective divisor. By \cite[Lemma 3.1]{C-hypersurface-lct}, $(X,D)$ is lc and indeed by \cite[Proposition 5]{Puk-mult-bound}, we have $\mult_x D\le 1$ for all but finitely many $x\in X$, hence by \cite[3.14.1]{Kol-sing-of-pairs}, $(X,D)$ has canonical singularities outside a finite number of points. It follows that every lc center of $(X,D)$ is either a divisor on $X$ or an isolated point.

Let $M\sim_\bQ -K_X$ be a movable boundary. By \cite[Proposition 5]{Puk-mult-bound} again, we have $\mult_x (M^2) \le 4$ outside a subset $Z\subseteq X$ of dimension at most $1$, hence by \cite{dFEM-mult-and-lct}, $(X,M)$ is lc outside $Z$. On the other hand by the main result of \cite{Puk-sing-index-two}, the only possible center of maximal singularities of $(X,M)$ is a linear section of $X$ codimension $2$. It follows that $(X,M)$ is lc and every lc center of $(X,M)$ is a linear section of codimension two.

Hence $X$ is K-semistable by Theorem \ref{thm:main}. Suppose that it is strictly K-semistable, then by Corollary \ref{cor:equality} there exists a movable boundary $M\sim_\bQ -K_X$ and an effective divisor $D\sim_\bQ H$ such that $(X,M)$ and $(X,D)$ has a common lc center. But by the previous analysis, this is impossible as the lc centers of $(X,M)$ and $(X,D)$ always have different dimension. Therefore $X$ is K-stable and the proof is complete.
\end{proof}

\begin{proof}[Proof of Theorem \ref{thm:alpha>=1/(n+1)}]
It suffices to show that $\lct(X;D)\ge\frac{1}{n+1}$ for every $D\sim_\bQ -K_X$. We may assume that $D$ is irreducible. Since $\mathrm{Cl}(X)$ is generated by $-K_X$, we have $\mult_{\eta}D\le 1$ where $\eta$ is the generic point of $D$. It follows that $(X,D)$ is log canonical in codimension one \cite[(3.14.1)]{Kol-sing-of-pairs}, hence the multiplier ideal $\cJ(X,(1-\epsilon)D)$ (where $0<\epsilon\ll 1$) defines a subscheme of codimension at least two. By Nadel vanishing,
\[H^i(X,\cJ(X,(1-\epsilon)D)\otimes \cO_X(-rK_X))=0
\]
for every $i>0$ and $r\ge0$, therefore by Castelnuovo-Mumford regularity (see e.g. \cite[\S 1.8]{positivity-1}), the sheaf $\cJ(X,(1-\epsilon)D)\otimes \cO_X(-nK_X)$ is generated by its global sections and we get a movable linear system 
\[\cM=|\cJ(X,(1-\epsilon)D)\otimes \cO_X(-nK_X)|.
\]
Suppose that $\lct(X;D)<\frac{1}{n+1}$, let $E$ be an exceptional divisor over $X$ that computes it, let $A=A_X(E)$ and let $\pi:Y\rightarrow X$ be a projective birational morphism such that the center of $E$ on $Y$ is a divisor, then $A\in \bZ$ (since $-K_X$ is Cartier by assumption) and we have $d=\ord_E(D)>(n+1)A$ and $\cJ(X,(1-\epsilon)D)\subseteq\pi_*\cO_Y((A-1-\lfloor (1-\epsilon)d \rfloor)E)\subseteq \pi_*\cO_Y(-(nA+1)E)$. It follows that $\ord_E(\cM)\ge nA+1$, hence for the movable boundary $M=\frac{1}{n}\cM\sim_\bQ -K_X$ we have $\ord_E(M)>A$ and $(X,M)$ is not log canonical, violating our assumption. Thus $\lct(X;D)\ge\frac{1}{n+1}$ and we are done.
\end{proof}

Finally we prove the separatedness statement (Theorem \ref{thm:separatedness}). For this we recall the following criterion:

\begin{lem} \label{lem:unique filling}
Let $f:X\rightarrow C$, $g:Y\rightarrow C$ be flat families of $\bQ$-Fano varieties over a smooth pointed curve $0\in C$ with central fibers $X_0$ and $Y_0$. Assume that $K_X$ and $K_Y$ are $\bQ$-Cartier and let $D_X\sim_\bQ -K_X$, $D_Y\sim_\bQ -K_Y$ be effective divisors not containing $X_0$ or $Y_0$. Assume that there exists an isomorphism 
\[\rho:(X,D_X)\times_C C^\circ \cong (Y,D_Y)\times_C C^\circ\]
over $C^\circ=C\backslash 0$, that $(X_0,D_X|_{X_0})$ is klt and $(Y_0,D_Y|_{Y_0})$ is lc. Then $\rho$ extends to an isomorphism $(X,D_X)\cong (Y,D_Y)$.
\end{lem}

\begin{proof}
This follows from the exact same proof of \cite[Theorem 5.2]{LWX}.
\end{proof}

\begin{proof}[Proof of Theorem \ref{thm:separatedness}]
Since birationally superrigid Fano varieties have terminal singularities, $K_X$ and $K_Y$ are $\bQ$-Cartier by \cite[Proposition 3.5]{dFH-deform-Fano}. Hence the result follows from Theorem \ref{thm:noether-fano} and the following more general statement.
\end{proof}

\begin{lem}
Let $f:X\rightarrow C$, $g:Y\rightarrow C$ be flat families of $\bQ$-Fano varieties over a smooth pointed curve $0\in C$ that are isomorphic over $C^\circ=C\backslash 0$. Let $X_0$ and $Y_0$ be their central fibers. Assume that
\begin{enumerate}
    \item $K_X$ and $K_Y$ are $\bQ$-Cartier.
    \item For every movable boundary $M_X\sim_\bQ -K_{X_0}$, $(X_0,M_X)$ is klt.
    \item For every movable boundary $M_Y\sim_\bQ -K_{Y_0}$, $(Y_0,M_Y)$ is lc.
\end{enumerate}
Then $X\cong Y$ over $C$.
\end{lem}

\begin{proof}
By assumption $X$ is birational to $Y$ over $C$. Let $m$ be a sufficiently large and divisible integer and let $D_1\in |-mK_{X_0}|$, $D_2\in |-mK_{Y_0}|$ be general divisors in the corresponding linear system. Choose effective divisors $D_{X,1}\sim_\bQ -mK_X$, $D_{Y,2}\sim_\bQ -mK_Y$ not containing $X_0$ or $Y_0$ such that $D_{X,1}|_{X_0}=D_1$ and $D_{Y,2}|_{Y_0}=D_2$. Let $D_{Y,1}$ and $D_{X,2}$ be their strict transforms to the other family. Let $D'_1=D_{Y,1}|_{Y_0}$ and $D'_2=D_{X,2}|_{X_0}$. Let $\cM_X$ be the linear system spanned by $D_{X,1}$ and $D_{X,2}$ and let $M_X=\frac{1}{m}\cM_X\sim_\bQ -K_X$. Similar we have $\cM_Y$ and $M_Y\sim_\bQ -K_Y$. As $D_1$ and $D_2$ are general, $D_1$ and $D'_2$ have no common components, hence the restriction of $M_X$ to $X_0$ is still a movable boundary and therefore by our second assumption, $(X_0,M_X|_{X_0})$ is klt. Similarly, $(Y_0,M_Y|_{Y_0})$ is lc and we conclude by Lemma \ref{lem:unique filling}.
\end{proof}

\bibliography{ref}
\bibliographystyle{alpha}

\end{document}